\documentclass[12pt]{article}
\usepackage{amsfonts}
\usepackage{amsmath}
\usepackage{amssymb}

\newtheorem{theorem}{Theorem}

\newtheorem{lemma}[theorem]{Lemma}

\newtheorem{proposition}[theorem]{Proposition}

\newenvironment{proof}[1][Proof]{\noindent\textbf{#1.} }{\ \rule{0.5em}{0.5em}}
\input{tcilatex}
\begin{document}

\author{Marek Galewski}
\title{A critical point theorem on a closed ball and some applications to
boundary value problems}
\maketitle

\begin{abstract}
We consider a functional being a difference of two differentiable convex
functionals on a closed ball. Existence and multiplicity of critical points
is investigated. Some applications are given.
\end{abstract}

\textbf{Keywords: }critical point, multiplicity, convexity

\textbf{Mathematics Subject Classification: }49J27, 39A12, 39A10

\section{Introduction}

In this paper we are concerned with the existence of a critical point to a
differentiable functional on a closed ball. With additional assumptions
pertaining to the mountain geometry we are in position to get the existence
of another distinct critical point as well. In a finite dimensional setting
the third critical point can also be obtained.

Let us introduce the space setting and the structure condition required on
the action functional. We assume that

\begin{description}
\item[H1] $E$ is a real reflexive Banach space compactly embedded into
another reflexive Banach space $Z;$

\item[H2] $\Phi :E\rightarrow 
%TCIMACRO{\U{211d} }%
%BeginExpansion
\mathbb{R}
%EndExpansion
,$ $H:Z\rightarrow 
%TCIMACRO{\U{211d} }%
%BeginExpansion
\mathbb{R}
%EndExpansion
$\ are Fr\'{e}chet differentiable convex functionals with derivatives $%
\varphi :E\rightarrow E^{\ast },$ $h:Z\rightarrow Z^{\ast }$ respectively;

\item[\textbf{H3}] there exist constants $\alpha >1,\gamma >0$ such that%
\begin{equation*}
\gamma \left\Vert v\right\Vert ^{\alpha }\leq \left\langle \varphi \left(
v\right) ,v\right\rangle _{E^{\ast },E}\text{ for all }v\in E\text{.}
\end{equation*}
\end{description}

We denote by $c>0$ the embedded constant, i.e. 
\begin{equation*}
\left\Vert v\right\Vert _{Z}\leq c\left\Vert v\right\Vert _{E}\text{ for all 
}v\in E
\end{equation*}%
and let $B_{\rho }\subset E$ be closed ball centered at $0$ in $E$ with
radius $\rho $. Assuming additionally that functional $x\rightarrow
\left\Vert h\left( x\right) \right\Vert _{Z^{\ast }}$ is bounded from above
on $B_{\rho }$ we will determine such a value $\lambda ^{\ast }>0$ that for
each $\lambda $ in the interval $\left( 0,\lambda ^{\ast }\right] $ the
corresponding Euler action functional $J:E\rightarrow \mathbb{%
%TCIMACRO{\U{211d} }%
%BeginExpansion
\mathbb{R}
%EndExpansion
}$%
\begin{equation*}
J\left( u\right) =\Phi \left( u\right) -\lambda H\left( u\right) 
\end{equation*}%
has a critical point on $B_{\rho }$. This implies the solvability of 
\begin{equation}
\varphi \left( u\right) =\lambda h\left( u\right) \text{, }u\in E
\label{pk_row1}
\end{equation}%
which is the Euler-Lagrange equation for $J$. Note that $u$ need not belong
to the interior of the ball $B_{\rho }$ and therefore we cannot use the
classical variational tool such as Ekelenad's variational principle in order
to demonstrate that the minimizer is a critical point. Such an approach is
used in \cite{BerJebMawhin} for a functional satisfying the PS-condition and
considered on a closed ball. If a minimizer is located in the interior of
this ball then it be proved via the Ekelenad's variational principle that it
is a critical point. When the PS-condition is not assumed, for example the
Fenchel-Young transform can be applied to prove that a minimizer is a
critical point. This approach is sketched in \cite{No1} and further
developed in several papers, see for example \cite{MGapll} and references
therein. In the present work we use only basic convexity and concavity to
the critical point together with some variational techniques based on the
Weierstrass Theorem.

Summarizing our approach: the first critical point (which lies in the ball,
perhaps on the boundary of the ball) is obtained through the Weierstrass
Theorem, direct method of the calculus of variations and convexity
relations, while the second critical point, under assumption that the
PS-condition is satisfied, is obtained with the aid of a general type of a
Mountain Pass Lemma. In a finite dimensional case, we obtain a third
critical point through direct maximization.

We provide applications to elliptic second order partial differential
equations and their discrete analogons put in the form of an algebraic
system.

For a background on variational methods we refer to \cite{KRV} for
differential equations to \cite{KufnerFucik}, \cite{Ma} while for a
background on difference equations to \cite{agarwalBOOK}. As for the
algebraic systems serving as example in this work, the literature is very
rich and we mention the following sources, \cite{MOLICALG2}, \cite{molicaalg}%
, \cite{YanZhangAPMacc}. The ideas connected with three critical point
theorems - different from those used in this work - are to be found for
example in \cite{bon1}, \cite{ric1}. Let us mention \cite{precup}, \cite%
{BerJebMawhin} for some recent results concerning a general type of critical
point theorem on a bounded set (with the PS-condition which we do not need),
and to \cite{serban2} for some applications of multiple critical point
theorem from \cite{BerJebMawhin} to discrete Neumann anisotropic problems.
Note that in \cite{precup} the bounded critical point theorem due to
Schechter is investigated, so the setting is in a Hilbert space, while in 
\cite{BerJebMawhin} it is a Banach space. The application of another type of
critical point on closed sets has just been developed by Marano, see \cite%
{marano1}, and to \cite{marano2} for applications to differential
inclusions, and also some earlier result \cite{MA}.

We provide necessary mathematical prerequisites which are needed for the
proof of the main multiplicity result.

Functional $J:E\rightarrow \mathbb{%
%TCIMACRO{\U{211d} }%
%BeginExpansion
\mathbb{R}
%EndExpansion
}$ satisfies the Palais-Smale condition (PS-condition for short) if every
sequence $(u_{n})$ such that $\{J(u_{n})\}$ is bounded and $J^{\prime
}(u_{n})\rightarrow 0$, has a convergent subsequence.

\begin{lemma}
\label{MPT}(\textbf{Mountain Pass Lemma, MPL Lemma}) Let $E$ be a Banach
space and assume that $J\in C^{1}(E,\mathbb{%
%TCIMACRO{\U{211d} }%
%BeginExpansion
\mathbb{R}
%EndExpansion
})$ satisfies the PS-condition. Let $S$ be a closed subset of $E$ which
disconnects $E$. Let $x_{0}$ and $x_{1}$ be points of $E$ which are in
distinct connected components of $E\backslash S$. Suppose that $J$ is
bounded below in $S$, and in fact the following condition is verified for
some $b$ 
\begin{equation}
\inf_{x\in S}J(x)\geq b\text{ and }\max \{J(x_{0}),J(x_{1})\}<b\text{.}
\label{condMPT}
\end{equation}%
If we denote by $\Gamma $ the family of continuous paths $\gamma
:[0,1]\rightarrow E$ joining $x_{0}$ and $x_{1},$ then%
\begin{equation*}
c:=\underset{\gamma \in \Gamma }{\inf }\underset{s\in \lbrack 0,1]}{\max }%
J(\gamma (s))\geq \max \{J(x_{0}),J(x_{1})\}>-\infty
\end{equation*}%
is a critical value and $J$ has a non-zero critical point $x$ at level $c.$
\end{lemma}

\section{A critical point theorem}

We begin with some general result which generalizes the main result from 
\cite{MGapll}

\begin{theorem}
\label{general}Assume that \textbf{H1},\textbf{\ H2} are satisfied. Fix some 
$\lambda ^{\ast }>0$ and let $u,v\in E$ be such that%
\begin{equation}
J\left( u\right) \leq J\left( v\right) \text{ and }\varphi \left( v\right)
=\lambda ^{\ast }h\left( u\right)  \label{zal}
\end{equation}%
Then $u$ is a critical point to $J$, and thus it solves (\ref{pk_row1}).
\end{theorem}

\begin{proof}
The proof follows by simple calculations pertaining to convexity of $\Phi $
and concavity of $-H$. Note that $J\left( u\right) \leq J\left( v\right) $
implies 
\begin{equation*}
\Phi \left( u\right) -\Phi \left( v\right) \leq -\lambda H\left( v\right)
-\left( -\lambda H\left( u\right) \right)
\end{equation*}%
so by standard inequalities following from definition of convexity of $\Phi $
at $v$ and concavity of $-H$ at $u$ we obtain 
\begin{equation*}
\left\langle \varphi \left( v\right) ,u-v\right\rangle \leq \Phi \left(
u\right) -\Phi \left( v\right) \leq -\lambda H\left( v\right) -\left(
-\lambda H\left( u\right) \right) \leq \left\langle \lambda h\left( u\right)
,u-v\right\rangle .\bigskip
\end{equation*}%
The above and the equality $\varphi \left( v\right) =\lambda h\left(
u\right) $ provide that $\Phi \left( u\right) =\left\langle \varphi \left(
v\right) ,u-v\right\rangle +\Phi \left( v\right) $. So from this relation
and by convexity again we have%
\begin{equation*}
\left\langle \varphi \left( v\right) ,v-u\right\rangle =\Phi \left( v\right)
-\Phi \left( u\right) \geq \left\langle \varphi \left( u\right)
,v-u\right\rangle .
\end{equation*}%
This means that both $\varphi \left( v\right) $ and $\varphi \left( u\right) 
$ are the elements of a subdifferential of $\Phi $ at $u$. Since, by
differentiability and convexity, this is a singleton, \cite{Ma}, we get that 
$\varphi \left( v\right) =\varphi \left( u\right) $. This by the equation in
(\ref{zal}) we see that $u$ is a critical point.
\end{proof}

However Theorem \ref{general}Some special cases of Theorem \ref{general} can
now be stated as follows. We make precise assumptions which lead to have
relations (\ref{zal}) satisfied.

\begin{theorem}
\label{multo}Let $E$ be a infinite dimensional reflexive Banach space.
Assume that \textbf{H1}-\textbf{H3} are satisfied. Take any $\rho >0$. 
\newline
(i) Assume that functional $x\rightarrow \left\Vert h\left( x\right)
\right\Vert _{Z^{\ast }}$ is bounded from above on $B_{\rho }$. \newline
Then there exists $\lambda ^{\ast }>0$ such that for each $\lambda \in
\left( 0,\lambda ^{\ast }\right] \ $there exist $u\in B_{\rho }$ with%
\begin{equation}
J\left( u\right) =\inf_{x\in B_{\rho }}J\left( x\right)  \label{Ju_mniej}
\end{equation}%
and such that $u$ is a critical point to $J$, and thus it solves (\ref%
{pk_row1}). If for some $v\in B_{\rho }$ it holds that $J\left( v\right) <0,$
$J\left( 0\right) =0$, then $u$ is non-trivial.\newline
(ii) Assume additionally that for all $\lambda \in \left( 0,\lambda ^{\ast }%
\right] \ $\newline
\ \ \ \ \ \ (ii a) $J$ satisfies the PS-condition, \newline
$\ \ \ \ \ \ $\ (ii b) $J\left( u\right) <\inf_{x\in \partial B_{\rho
_{1}}}J\left( x\right) $ for some $\rho _{1}>\left\Vert u\right\Vert $ ,%
\newline
\ \ \ \ \ \ (ii c) \ there exists $w\in E$ with $\lim_{t\rightarrow \infty
}J\left( tw\right) =-\infty $. \newline
Then for all $\lambda \in \left( 0,\lambda ^{\ast }\right] \ $functional $J$
has two critical points, namely $u$ and another non-zero critical point $z$
different from $u$.
\end{theorem}

\begin{proof}
Denote by $\beta >0$ the upper bound of functional $x\rightarrow \left\Vert
h\left( x\right) \right\Vert _{Z^{\ast }}$ on $B_{\rho }$. Put $\lambda
^{\ast }=\frac{\gamma \rho ^{\alpha -1}}{\beta c}$ and fix $\lambda \leq
\lambda ^{\ast }$. Consider $J$ on $B_{\rho }$. Observe that $J$ is
sequentially weakly l.s.c. on $B_{\rho }$. Indeed, let $\left( u_{n}\right) $
be a sequence from $B_{\rho }$. Then we can assume that $\left( u_{n}\right) 
$ is weakly convergent in $E$ and strongly in $Z$ to some $\overline{u}\in
B_{\rho }$, so $H\left( u_{u}\right) $ converges to $H\left( \overline{u}%
\right) $.\ Since $\Phi $ is weakly l.s.c. as a convex functional, we see
that $J$ is weakly l.s.c. on $B_{\rho }$. Since $B_{\rho }$ is weakly
compact some $u$ exists such that (\ref{Ju_mniej}) holds. \bigskip

Now consider on $E$ functional $J_{1}$ given by the formula 
\begin{equation*}
J_{1}\left( x\right) =\Phi \left( x\right) -\left\langle h\left( u\right)
,x\right\rangle _{Z^{\ast },Z}.
\end{equation*}%
Since $\Phi $ is weakly l.s.c. and is coercive, so is $J_{1}$ and therefore
we get the existence of an argument of a minimum to $J_{1}$ over $E,$ which
we denote by $v$. Obviously $v$ is a critical point to $J_{1}$ and so for
any $x\in E$%
\begin{equation}
\left\langle \varphi \left( v\right) ,x\right\rangle _{E^{\ast },E}-\lambda
\left\langle h\left( u\right) ,x\right\rangle _{Z^{\ast },Z}=0.
\label{equ_weak}
\end{equation}%
This means that $h$ solves 
\begin{equation}
\varphi \left( v\right) =\lambda h\left( u\right)  \label{equequequ}
\end{equation}%
in the weak sense. Observe $v$ belongs to $B_{\rho }.$ Indeed, put $x=v$ in (%
\ref{equ_weak}). Thus%
\begin{equation*}
\begin{array}{l}
\gamma \left\Vert v\right\Vert ^{\alpha }\leq \left\langle \varphi \left(
v\right) ,v\right\rangle _{E^{\ast },E}=\lambda \left\langle h\left(
u\right) ,v\right\rangle _{Z^{\ast },Z}\leq \bigskip \\ 
\lambda \left\Vert h\left( u\right) \right\Vert _{Z^{\ast }}\left\Vert
v\right\Vert _{Z}\leq \lambda \beta c\left\Vert v\right\Vert _{E}.%
\end{array}%
\end{equation*}%
Therefore $\left\Vert h\right\Vert ^{\alpha -1}\leq \lambda \frac{\beta c}{%
\gamma }\leq \rho ^{\alpha -1}$ and $h\in B_{\rho }.$

The proof that $u$ is a critical point follows from Theorem \ref{general}
since $J\left( u\right) \leq J\left( v\right) $ and since (\ref{equequequ})
holds. When $J\left( v\right) <0$, then also $J\left( u\right) <0$ and the
assertion that $u$ is nontrivial follows since $J\left( 0\right) =0$.\bigskip

In order to prove part (ii), i.e. in order to get the second critical point,
we will use Lemma \ref{MPT}. Since $\lim_{t\rightarrow \infty }J\left(
tw\right) =-\infty $, so there exists some $w_{1}$ such that 
\begin{equation*}
J\left( w_{1}\right) \leq \inf_{x\in B_{\rho }}J\left( x\right) <\inf_{x\in
\partial B_{\rho _{1}}}J\left( x\right) .
\end{equation*}%
Thus we have condition (\ref{condMPT}) satisfied taking $x_{0}=u$ and $%
x_{1}=w_{1}$. The existence of a second non-zero critical point readily
follows.
\end{proof}

We see that when $J\left( 0\right) =0$ condition\ (ii b) can be replaced by
the following\newline
$\ \ \ \ \ \ $\ (ii b) $\inf_{x\in \partial B_{\rho _{1}}}J\left( x\right)
>0 $ for some $\rho _{1}>0$

In a finite dimensional context, we get easily the existence of a third
critical point as follows

\begin{theorem}
\label{multo copy(1)}Let $E$ be a finite dimensional Banach. Assume that 
\textbf{H1}-\textbf{H3} are satisfied. Take any $\rho >0$. Then there exists 
$\lambda ^{\ast }>0$ such that for each $\lambda \in \left( 0,\lambda ^{\ast
}\right] \ $there exist $u\in B_{\rho }$ with%
\begin{equation*}
J\left( u\right) =\inf_{x\in B_{\rho }}J\left( x\right)
\end{equation*}%
and such that $u$ is a critical point to $J$, and thus it solves (\ref%
{pk_row1}). If for some $v\in B_{\rho }$ it holds that $J\left( v\right) <0,$
$J\left( 0\right) =0$, then $u$ is non-trivial.\newline
(ii) Assume additionally that \newline
\ \ \ \ \ \ (ii a) $J$ is anti-coercive \newline
$\ \ \ \ \ \ $\ (ii b) $J\left( u\right) <\inf_{x\in \partial B_{\rho
_{1}}}J\left( x\right) $ for some $\rho _{1}\geq \rho $ ,\newline
Then for any $\lambda \in \left( 0,\lambda ^{\ast }\right] \ $functional $J$
has at least three critical points, namely $u$ and two another critical
points one being a Mountain Pass point and the other the argument of a
maximum.
\end{theorem}

\begin{proof}
We define $\lambda ^{\ast }$ as in the proof of Theorem \ref{multo}. By the
Weierstrass Theorem condition (i) holds by continuity. Note that in a finite
dimensional setting an anti-coercive functional necessarily satisfies the
PS-condition and morevoer, there exists $w\in E$ with $\lim_{t\rightarrow
\infty }J\left( tw\right) =-\infty $. Thus the existence of two distinct
solutions, $u$ and some $z\neq 0$, follows by Theorem \ref{multo}. Since $J$
is anti-coercive and continuous it has an argument of a maximum over $E$
which we denote by $w$. Since $J$ is differentiable it follows that $w$ is a
critical point. Since 
\begin{equation*}
\max \left\{ J\left( z\right) ,J\left( u\right) \right\} \leq \sup_{x\in
E}J\left( x\right)
\end{equation*}%
we see that either $w$ is a third critical point distinct from the previous
ones or else there are infinitely many critical points at the level $J\left(
z\right) $.
\end{proof}

\section{Examples of applications}

In this section we apply our abstract results for nonlinear algebraic
systems being discretization of some elliptic problem and to their
continuous counterpart. The examples show that the discrete case is not only
less demanding, but also multiple solutions are obtained in a easier manner.
For example we do not need a type of A-R condition. It seems that due to
relatively mild conditions required in order to obtain at least one critical
point (in fact local growth conditions suffice), Theorem \ref{multo} would
apply for various boundary value problems. We note that if one wants to
obtain a solution of a mountain pass type then the second critical point
follows by simple assuming convexity of a potential of a RHS of the equation
under consideration which is not very demanding when A-R has been assumed.

\subsection{Application to the partial difference equations}

We will consider the system

\begin{equation}
\begin{array}{l}
\left[ u(i+1,j)-2u(i,j)+u(i-1,j)\right] +\left[ u(i,j+1)-2u(i,j)+u(i,j-1)%
\right] \bigskip \\ 
+\lambda f((i,j),u(i,j))=0,\bigskip \\ 
\text{ for all }i\in \{1,..,m\},j\in \{1,...,n\}\bigskip \\ 
u(i,0)=u(i,n+1)=0\text{ for all }i\in \{1,..,m\}\bigskip \\ 
u(0,j)=u(m+1,j)=0\text{ for all }j\in \{1,...,n\}%
\end{array}
\label{DP}
\end{equation}%
which serves as the discrete counterpart of the problem%
\begin{equation}
\begin{array}{l}
\frac{\partial ^{2}u}{\partial x^{2}}+\frac{\partial ^{2}u}{\partial y^{2}}%
+\lambda f((x,y),u(x,y))=0\bigskip \\ 
u(x,0)=u(x,n+1)=0,\text{ for all }x\in (0,m+1)\bigskip \\ 
u(0,y)=u(m+1,y)=0\text{ for all }y\in (0,n+1)%
\end{array}
\label{CP}
\end{equation}

Following some ideas from \cite{boyang}, we write (\ref{DP}) as a nonlinear
system which we further investigate. Let 
\begin{equation*}
A:=\left[ 
\begin{array}{ccccccccc}
L & -I_{m} & 0 & 0 & ... & 0 & 0 & 0 & 0 \\ 
-I_{m} & L & -I_{m} & 0 & ... & 0 & 0 & 0 & 0 \\ 
0 & -I_{m} & L & -I_{m} & ... & 0 & 0 & 0 & 0 \\ 
0 & 0 & -I_{m} & L & ... & 0 & 0 & 0 & 0 \\ 
... & ... & ... & ... & ... & ... & ... & ... & ... \\ 
0 & 0 & 0 & 0 & ... & L & -I_{m} & 0 & 0 \\ 
0 & 0 & 0 & 0 & ... & -I_{m} & L & -I_{m} & 0 \\ 
0 & 0 & 0 & 0 & ... & 0 & -I_{m} & L & -I_{m} \\ 
0 & 0 & 0 & 0 & ... & 0 & 0 & -I_{m} & L%
\end{array}%
\right]
\end{equation*}%
where $I_{m}$ is identity matrix of order $m$ and $L$ is $m\times m$ matrix
defined by%
\begin{equation*}
L:=\left[ 
\begin{array}{ccccccccc}
4 & -1 & 0 & 0 & ... & 0 & 0 & 0 & 0 \\ 
-1 & 4 & -1 & 0 & ... & 0 & 0 & 0 & 0 \\ 
0 & -1 & 4 & -1 & ... & 0 & 0 & 0 & 0 \\ 
0 & 0 & -1 & 4 & ... & 0 & 0 & 0 & 0 \\ 
... & ... & ... & ... & ... & ... & ... & ... & ... \\ 
0 & 0 & 0 & 0 & ... & 4 & -1 & 0 & 0 \\ 
0 & 0 & 0 & 0 & ... & -1 & 4 & -1 & 0 \\ 
0 & 0 & 0 & 0 & ... & 0 & -1 & 4 & -1 \\ 
0 & 0 & 0 & 0 & ... & 0 & 0 & -1 & 4%
\end{array}%
\right] .
\end{equation*}%
Matrix $A$ is positive definite, see \cite{boyang}. Thus problem (\ref{DP})
can be rewritten as 
\begin{equation}
Au=\lambda f(u),  \label{problem}
\end{equation}%
with%
\begin{equation*}
\begin{array}{l}
u=(u(1,1),...,u(m,1);u(1,2),...,u(m,2);u(1,n),...,u(m,n))^{T},\bigskip \\ 
f(u):=(\left( f((1,1),u(1,1)),...,f((m,1),u(m,1)),\right. \bigskip \\ 
\left. f((1,2),u(1,2)),...,f((m,2),u(m,2)),\right. \bigskip \\ 
\left. f((1,n),u(1,n)),...,f((m,n),u(m,n\right) )^{T}.%
\end{array}%
\end{equation*}%
With $f$ being a continuous function, solutions to (\ref{problem})
correspond in a one to one manner to critical points of a functional $J:%
\mathbb{%
%TCIMACRO{\U{211d} }%
%BeginExpansion
\mathbb{R}
%EndExpansion
}^{n}\times \mathbb{%
%TCIMACRO{\U{211d} }%
%BeginExpansion
\mathbb{R}
%EndExpansion
}^{m}$ $\rightarrow \mathbb{%
%TCIMACRO{\U{211d} }%
%BeginExpansion
\mathbb{R}
%EndExpansion
}$%
\begin{equation*}
J(u)=\frac{1}{2}(u,Au)-\lambda \sum_{i=1}^{m}\sum_{j=1}^{n}F((i,j),u(i,j)),
\end{equation*}%
where 
\begin{equation*}
F((i,j),u(i,j)):=\int_{0}^{u(i,j)}f((i,j),v)dv.
\end{equation*}%
By $\alpha _{1},\alpha _{2},...,\alpha _{mn}$ we denote the eigenvalues of $%
A $ ordered as%
\begin{equation}
0<\alpha _{1}<\alpha _{2}<...\leq \alpha _{mn}.  \label{ALFA1}
\end{equation}%
The assumptions which we impose read$\bigskip $

\begin{description}
\item[\textbf{H4}] $f((i,j),\cdot ):\mathbb{%
%TCIMACRO{\U{211d} }%
%BeginExpansion
\mathbb{R}
%EndExpansion
}\rightarrow \mathbb{%
%TCIMACRO{\U{211d} }%
%BeginExpansion
\mathbb{R}
%EndExpansion
}$ is continuous for all $i\in \{1,..,m\},j\in \{1,...,n\}$ and there exist
constants $\mu >2$, $c_{1}>0,$ $c_{2}\in \mathbb{%
%TCIMACRO{\U{211d} }%
%BeginExpansion
\mathbb{R}
%EndExpansion
}$, $d>0$ 
\begin{equation*}
F((i,j),x)\geq c_{1}|x|^{\mu }+c_{2}
\end{equation*}%
for all $i\in \{1,..,m\},j\in \{1,...,n\}$ and all $\left\vert x\right\vert
\geq d$;$\bigskip $

\item[H5] function $x\rightarrow F\left( (i,j),x\right) $ is convex on $%
%TCIMACRO{\U{211d} }%
%BeginExpansion
\mathbb{R}
%EndExpansion
$ for all $i\in \{1,..,m\},j\in \{1,...,n\}$.
\end{description}

The assumptions employed here are not very restrictive. There are many
functions satisfying both \textbf{H4} and \textbf{H5. }See for example $%
F(k,x)=c_{1}|x|^{\mu }+c_{2}$ with $\mu >2$ and even. Then we arrive at the
following theorem$\hfill $

\begin{theorem}
Assume that conditions \textbf{H4}-\textbf{H5} are satisfied. There exists $%
\lambda ^{\ast }>0$ such that for all $0<\lambda \leq \lambda ^{\ast }$
problem (\ref{DP}) has at least three nontrivial solutions.
\end{theorem}

\begin{proof}
We see that $E=Z=%
%TCIMACRO{\U{211d} }%
%BeginExpansion
\mathbb{R}
%EndExpansion
^{m}\times 
%TCIMACRO{\U{211d} }%
%BeginExpansion
\mathbb{R}
%EndExpansion
^{n}$. In this case $c=1$. Observe that $\gamma =\alpha _{1}$, see (\ref%
{ALFA1}) and $\alpha =2$ since 
\begin{equation*}
(u,Au)\geq \alpha _{1}\left\vert u\right\vert ^{2}\text{.}
\end{equation*}%
It follows by a direct computation that for any $\lambda >0$ functional $J$
is anti-coercive, i.e. $J(x)\rightarrow -\infty $ as $||x||\rightarrow
+\infty $. So there is $z_{1}\in E$ such that $J\left( z_{1}\right) <0$.
Take $\rho \geq \left\vert z_{1}\right\vert $. We denote by $\beta $ the
maximal value of a functional $x\rightarrow \sqrt{\sum_{i=1}^{m}%
\sum_{j=1}^{n}f^{2}((i,j),x(i,j))}$ on $B_{\rho }$ which is finite by a
Weierstrass Theorem and we put $\lambda ^{\ast }=\frac{\alpha _{1}\rho }{%
\beta }.$ Note that condition (ii b) of Theorem \ref{multo copy(1)}\ follows
by anti-coercivity.\bigskip
\end{proof}

\subsection{Applications to partial differential equations}

In this section we consider problems similar to (\ref{CP}), namely

\begin{equation}
\left\{ 
\begin{array}{c}
-\Delta u(x)=\lambda f(x,u(x)),\text{ }\left. u\right\vert _{\partial \Omega
}=0\bigskip \\ 
u\in W_{0}^{1,2}(\Omega ),\text{ }%
\end{array}%
\right.  \label{M1R}
\end{equation}%
with a numerical parameter $\lambda >0$ and where $\Omega \subset \mathbb{%
%TCIMACRO{\U{211d} }%
%BeginExpansion
\mathbb{R}
%EndExpansion
}^{n}$, $n\geq 2$, $\Omega $ is a smooth bounded region. Let $F\left(
x,v\right) =\int_{0}^{v}f\left( x,\tau \right) d\tau $ for $a.e.$ $x\in
\Omega $. We will assume that

\begin{description}
\item[ \textbf{H6}] $f:\Omega \times 
%TCIMACRO{\U{211d} }%
%BeginExpansion
\mathbb{R}
%EndExpansion
\rightarrow 
%TCIMACRO{\U{211d} }%
%BeginExpansion
\mathbb{R}
%EndExpansion
$ is a Caratheodory function;

\item[\textbf{H7}] \textit{there exists a constant }$\theta >2$\textit{\
such that for }$v\in 
%TCIMACRO{\U{211d} }%
%BeginExpansion
\mathbb{R}
%EndExpansion
$, $v\neq 0$ and $a.e.$\textit{\ }$x\in \Omega $\textit{\ }%
\begin{equation*}
0<\theta F\left( x,v\right) \leq vf\left( x,v\right) ;\bigskip
\end{equation*}

\item[\textbf{H8}] \textit{there exist constants }$\beta _{1},\eta >0,\beta
_{2}\geq 0$\textit{\ with }$\eta >2$ and\textit{\ such that for all }$v\in 
%TCIMACRO{\U{211d} }%
%BeginExpansion
\mathbb{R}
%EndExpansion
$\textit{\ and }$a.e.$\textit{\ }$x\in \Omega $\textit{\ }%
\begin{equation*}
\left\vert f\left( x,v\right) \right\vert \leq \beta _{1}\left\vert
v\right\vert ^{\eta -1}+\beta _{2};\bigskip
\end{equation*}

\item[\textbf{H9}] $\lim_{v\rightarrow 0}\frac{\left\vert f\left( x,v\right)
\right\vert }{\left\vert v\right\vert }=0$ \textit{uniformly for\ a.e. }$%
x\in \Omega $;\bigskip

\item[\textbf{H10}] \textit{function }$v\rightarrow F\left( x,v\right) $%
\textit{\ is convex on }$%
%TCIMACRO{\U{211d} }%
%BeginExpansion
\mathbb{R}
%EndExpansion
$\textit{\ for a.e. }$x\in \Omega $. \bigskip
\end{description}

We see that the action functional $J:W_{0}^{1,2}(\Omega )\rightarrow 
%TCIMACRO{\U{211d} }%
%BeginExpansion
\mathbb{R}
%EndExpansion
$ given by%
\begin{equation}
J\left( u\right) =\frac{1}{2}\int_{\Omega }\left\vert \nabla u\left(
x\right) \right\vert ^{2}dx-\lambda \int_{\Omega }F\left( x,u\left( x\right)
\right) dx  \label{FIR}
\end{equation}%
is continuously G\^{a}teaux differentiable. Thus it is a $C^{1}$ functional.
Weak solutions to (\ref{M1R}) i.e. a functions $u$ satisfying%
\begin{equation*}
\int_{\Omega }\nabla u\left( x\right) \nabla v\left( x\right) dx=\lambda
\int_{\Omega }f\left( x,u\left( x\right) \right) v\left( x\right) dx\text{
for all }v\in W_{0}^{1,2}(\Omega )
\end{equation*}%
are critical points to $J$. From \cite{jabri} we get the two lemmas
concerning the mountain geometry for (\ref{M1R}).

\begin{lemma}
\label{LemPS}Suppose that \textbf{H6}-\textbf{H8} hold. Then for any $%
\lambda >0$ the functional $J$ given by (\ref{FIR}) satisfies the
PS-condition.
\end{lemma}

\begin{lemma}
\label{LemMntPasGeom}Suppose that \textbf{H6}-\textbf{H9} hold. Then for any 
$\lambda >0$ \ there exist numbers $\kappa ,\xi >0$ such that $J\left(
u\right) \geq \xi $ for all $u\in W_{0}^{1,2}\left( \Omega \right) $ with $%
\left\Vert u\right\Vert _{W_{0}^{1,2}}=\kappa $. Moreover, there exists an
element $z\in W_{0}^{1,2}\left( \Omega \right) $ with $\left\Vert
z\right\Vert _{W_{0}^{1,2}}>\kappa $ and such that $J\left( z\right) <0$.
\end{lemma}

Using Mountain Pass Lemma \ref{MPT} and Lemmas \ref{LemPS} and \ref%
{LemMntPasGeom}, we get the following

\begin{proposition}
Suppose that \textbf{H6}-\textbf{H9} hold. Then for any $\lambda >0$ problem
(\ref{M1R}) has at least one nontrivial solution.
\end{proposition}

Concerning the multiple solutions we have the main result of this section
where we need only assume that $F$ is convex in addition to assumptions
leading to a mountain pass solution.

\begin{theorem}
Assume that conditions \textbf{H6}-\textbf{H10} are satisfied. Then there
exists $\lambda ^{\ast }>0$ such that for all $0<\lambda \leq \lambda ^{\ast
}$ problem (\ref{M1R}) has at least two solutions.
\end{theorem}

\begin{proof}
We put $E=W_{0}^{1,2}\left( \Omega \right) $, $Z=L^{2}\left( \Omega \right) $%
. Here $\alpha =2,\gamma =1$ and the constant $c$ is the best constant in
the Poincar\'{e} inequality. Condition (i) of Theorem \ref{multo} is
satisfied by \textbf{H8.} Lemmas \ref{LemPS} and \ref{LemMntPasGeom} provide
condition (ii) of Theorem \ref{multo} and thus the application of this
theorem finishes the proof.
\end{proof}

\begin{tabular}{l}
Marek Galewski \\ 
Institute of Mathematics, \\ 
Technical University of Lodz, \\ 
Wolczanska 215, 90-924 Lodz, Poland, \\ 
marek.galewski@p.lodz.pl%
\end{tabular}


\begin{thebibliography}{99}
\bibitem{jabri} Y. Jabri, The mountain pass theorem. Variants,
generalizations and some applications. Encyclopedia of Mathematics and its
Applications, 95.
\end{thebibliography}
\end{document}